\documentclass[11pt]{amsart}
\usepackage{amsthm,amsmath,amssymb}
\newtheorem{thm}{Theorem}
\newtheorem{lem}[thm]{Lemma}
\newtheorem{cor}[thm]{Corollary}
\theoremstyle{remark}
\newtheorem{rem}[thm]{Remark}

\begin{document}

\title{On the characterization of Galois extensions}

\author{Meinolf Geck}
\date{}


\begin{abstract} We present a shortcut to the familiar characterizations
of finite Galois extensions, based on an idea from an earlier 
Monthly note by Sonn and Zassenhaus. \end{abstract}

\maketitle

Let $L\supseteq K$ be a field extension and $G:=\mbox{Aut}(L,K)$. We
assume throughout that $[L:K]<\infty$. One easily sees that then 
automatically $|G|<\infty$. The following result is an essential part 
of the usual development (and teaching) of Galois theory.

\begin{thm} \label{thm1} The following conditions are equivalent.
\begin{itemize}
\item[(a)] $|G|=[L:K]$. 
\item[(b)] $L$ is a splitting field for a polynomial $0\neq f\in K[X]$ 
which does not have multiple roots in~$L$.
\item[(c)] $K=\{y\in L\mid \sigma(y)=y\mbox{ for all $\sigma\in G$}\}$.
\end{itemize}
\end{thm}

If these conditions hold, then $L\supseteq K$ is called a {\em Galois
extension}. For example, condition (b) yields a simple criterion for an 
extension to be a Galois extension, and condition (c) is crucial for many 
applications of Galois theory. Combining all three immediately shows that, 
if $L\supseteq K$ is a Galois extension and $M$ is an intermediate field, 
then $L\supseteq M$ also is a Galois extension and so
\[M=\{y\in L\mid \sigma(y)=y\mbox{ for all $\sigma\in H$}\} \qquad
\mbox{where}\quad H:=\mbox{Aut}(L,M)\subseteq G,\]
which is a significant part of the ``Main Theorem of Galois Theory''.

Proofs of Theorem~\ref{thm1} often rely on Dedekind's Lemma on group 
characters and Artin's Theorem (see \cite[Chap.~II, \S F]{art1}), and some
results on normal and separable extensions. Some textbooks (e.g. \cite{art2},
\cite{St}) use the ``Theorem on primitive elements'' at an early stage 
to obtain a shortcut. It is the purpose of this note to point out that
there is a different shortcut which avoids using the existence 
of primitive elements and actually establishes this existence as a 
by-product (at least for Galois extensions). This only relies on a few 
basic results about fields (e.g., the degree formula and the uniqueness 
of splitting fields); no assumptions on the characteristic are required. 
The starting point is the following observation.

\begin{lem} \label{lem1} If $K\subsetneqq L$, then $L$ is not the union 
of finitely many fields $M$ such that $K\subseteq M\subsetneqq L$.
\end{lem}

\begin{proof} If $K$ is infinite, then each $M$ as above is a proper 
subspace of the $K$-vector space $L$, and it is well-known and easy to prove
that a finite-dimensional vector space over an infinite field is not the union 
of finitely many proper subspaces. 
Now assume that $K$ is finite. Then $L$ is also finite and so $|L|=p^n$ for 
some prime $p$. In this case, it is not enough just to argue with the
vector space structure. One could use the fact that the multiplicative 
group of $L$ is cyclic. Or one can argue as follows. Again, it is well-known
and easy to prove that, for every $m\leq n$, there is at most one subfield
$M\subseteq L$ such that $|M|=p^m$. (The elements of such a subfield are
roots of the polynomial $X^{p^m}-X\in L[X]$.)
So the total number of elements in $L$ which lie in proper subfields is 
at most $1+p+\cdots + p^{n-1}< p^n=|L|$, as desired.
\end{proof}

\begin{cor} \label{cor1} There exists an element $z\in L$ such that 
$\operatorname{Stab}_G(z)=\{\operatorname{id}\}$. 
\end{cor}

\begin{proof} If $G=\{\mbox{id}\}$, there is nothing to prove. Now assume
that $G\neq\{\mbox{id}\}$. For $\mbox{id}\neq\sigma\in G$ we set 
$M_\sigma:=\{y\in L\mid \sigma(y)=y\}$. Then $M_\sigma$ is a field such that 
$K\subseteq M_\sigma \subsetneqq L$. Now apply Lemma~\ref{lem1}.
\end{proof}

\begin{cor} \label{cor2} We always have $|G|\leq [L:K]$. If equality holds, 
then there is some $z\in L$ such that $L=K(z)$ and the minimal polynomial 
$\mu_z\in K[X]$ has only simple roots in $L$; furthermore, $L$ is a 
splitting field for $\mu_z$.
\end{cor}

\begin{proof} Let $z\in L$ be as in Corollary~\ref{cor1}. Let $G=\{\sigma_1,
\ldots,\sigma_m\}$. Then $\{\sigma_i(z)\mid 1\leq i\leq m\}$ has precisely
$m$ elements. Now $[L:K]\geq [K(z):K]=\deg(\mu_z)$. Since $\mu_z$ has 
coefficients in $K$, we have $\mu_z(\sigma_i(z))=\sigma_i(\mu_z(z))=0$ for 
all $i$. So $\mu_z$ has at least $m$ distinct roots; in particular,
$\deg(\mu_z)\geq m=|G|$. This shows that $[L:K]\geq |G|$. If $[L:K]=|G|$,
then all of the above inequalities must be equalities. This yields 
$L=K(z)$ and $\deg(\mu_z)=m$; in particular, $\mu_z=\prod_{i=1}^m
(X-\sigma_i(z))$ has only simple roots and $L$ is a splitting field 
for $\mu_z$.
\end{proof}

Corollary~\ref{cor2} shows the implication ``(a) $\Rightarrow$ (b)'' in 
Theorem~\ref{thm1} and also establishes the existence of a primitive element. 
Then the remaining implications in Theorem~\ref{thm1} are proved by 
standard arguments, which we briefly sketch:

\medskip
\noindent {\bf Proof of ``(a) $\Rightarrow$ (c)'':}  
Let $M:=\{y\in L\mid \sigma(y)=y\mbox{ for all $\sigma\in G$}\}$. Then
$M$ is a field such that $K\subseteq M\subseteq L$ and it is clear from
the definitions that $G=\mbox{Aut}(L,M)$. Hence, Corollary~\ref{cor2} shows
that $|G|\leq [L:M]\leq [L:K]$. Since (a) holds, this implies that
$[L:M]=[L:K]$ and so $M=K$. 

\medskip
\noindent {\bf Proof of ``(c) $\Rightarrow$ (b)'':}  Let $L=K(z_1,\ldots,
z_m)$ and form the set $B:=\{\sigma(z_i)\mid 1\leq i\leq m,\sigma\in G\}$. 
Since (c) holds, we have $f:=\prod_{z\in B} (X-z) \in K[X]$; furthermore, 
$L$ is a splitting field for $f$, and $f$ has no multiple roots. 

\medskip
\noindent {\bf Proof of ``(b) $\Rightarrow$ (a)'':} This relies on
a standard result on extending field isomorphisms (which is also used
to prove that any two splitting fields of a polynomial are isomorphic; see, 
e.g., \cite[Theorem~10]{art1}). Using this result and induction on $[L:K]$, 
it is a simple matter of book-keeping (no further theory required) to 
construct $[L:K]$ distinct elements of $G$; the details can be found, for 
example, in \cite[Chap.~14, (5.4)]{art2}. This shows that $|G|\geq [L:K]$, 
and Corollary~\ref{cor2} yields equality. 

\medskip
Once Theorem~\ref{thm1} and Corollary~\ref{cor2} are established, the
``Main Theorem of Galois Theory'' now follows rather quickly; see
\cite[Chap.~14, \S 5]{art2} or \cite[\S 9.3]{St}. 

\begin{rem} The idea of looking at elements of $L$ which do not lie in 
proper subfields is taken from~\cite{zass}.
\end{rem}


\bigskip
\noindent {IAZ -- Lehrstuhl f\"ur Algebra, Universit\"at Stuttgart,
Pfaffenwaldring 57, 70569 Stuttgart, Germany\\
meinolf.geck@mathematik.uni-stuttgart.de}

\end{document}